\documentclass{article}
\usepackage{url}
\usepackage{physics}
\usepackage{graphicx}
\usepackage{longtable}
\usepackage{array}
\usepackage{ifthen}
\usepackage{physics}
\usepackage[english]{babel}
\usepackage{amsthm}
\usepackage{amsfonts}
\usepackage{authblk}
\usepackage[backend=biber, style=alphabetic]{biblatex}

\newtheorem{theorem}{Theorem}

\newtheorem{lemma}[theorem]{Lemma}

\title{On one-sided series of Bessel functions of the~first~kind:\\
$\displaystyle\sum_{m=0}^{\infty} \qty(\pm 1)^m J_{Nm+p}(x)$.
}
\author[1,2]{Matwey Kornilov\footnote{\protect\url{matwey@sai.msu.ru}}}
\affil[1]{Lomonosov Moscow State University}
\affil[2]{HSE University}
\date{}

\begin{document}

\newcommand{\JC}[1]{\ifthenelse{\equal{#1}{0}}
    {\ensuremath{J_{c_{0}}}}
    {\ensuremath{J_{c_{0,#1}}}}
}
\newcommand{\JS}[1]{\ifthenelse{\equal{#1}{0}}
    {\ensuremath{J_{s_{0}}}}
    {\ensuremath{J_{s_{0,#1}}}}
}

\newcolumntype{R}{>{$\displaystyle}r<{$}}
\newcolumntype{C}{>{$\displaystyle}c<{$}}
\newcolumntype{L}{>{$\displaystyle}l<{$}}

\maketitle

\begin{abstract}
We present generic expressions for one-sided series of Bessel functions of the first kind:
$\displaystyle\sum_{m=0}^{\infty} J_{Nm+p}(x)$ and $\displaystyle\sum_{m=0}^{\infty} (-1)^m J_{Nm+p}(x)$.
The result is expressed in terms of incomplete Lipschitz-Hankel integrals of Bessel functions of the first kind, which can be represented as Kamp\'e de F\'eriet functions.
\end{abstract}

\section{Introduction}

Series of Bessel functions $J_{\nu}(x)$ are common objects of study in the mathematical literature.
Series and integrals tables such as~\cite{Gradshteyn2014},~\cite{Abramowitz1965}, and~\cite{Prudnikov1986} provide results for series of the form $\sum_{m=0}^{\infty} (\pm 1)^m J_{Nm+p}(x)$ for specific values of $N$ and $p$.

Nevertheless, a generic closed-form expression valid for arbitrary $N$ and $p$ does not appear to be readily available in the standard references.
Moreover, computer algebra systems often fail to produce a closed-form result for general parameter values.
Probably, the difficulty is related to the fact that the natural closed form involves incomplete Lipschitz-Hankel integrals of Bessel functions, which belong to a class of special functions that is less widely implemented.
The present paper is motivated by the recent work~\cite{Sung2022}, where generic closed-form expressions for bilateral series of Bessel functions were derived.

The paper is organized as follows.
In Section~\ref{sec:plus_series} we derive a generic expression for $\sum_{m=0}^{\infty} \qty(\pm 1)^m J_{Nm+p}(x)$.
We inspect properties of the result and provide numerical validation of the obtained expressions in Section~\ref{sec:numerical_evaluation}.
Section~\ref{sec:conclusion} contains concluding remarks.

\section{Multisectioned series of Bessel functions}
\label{sec:plus_series}

We begin by considering the convergence of the function sequence itself.
Although this is not directly related to the series under consideration, the result will be reused later to interchange the order of integration and the limit.

\begin{lemma}
\label{lem:J_conv}
The function sequence $\left\{J_{Nm+p}\right\}_m$ converges uniformly to $0$ as $m \to \infty$ for $x \in \mathbb{R}$ and any $N \in \mathbb{N}$, $p \in \mathbb{Z}$.
\end{lemma}

\begin{proof}
For any $x \in \mathbb{R}$ and $\nu > 0$, the Bessel function satisfies the following inequality:
\begin{equation}
|J_\nu(x)| \le b \nu^{-1/3},
\end{equation}
where $b$ is an absolute constant.~\cite{Landau2000} Applying this bound with $\nu = Nm+p$ yields
\begin{equation}
|J_{Nm+p}(x)| \le b (Nm+p)^{-1/3}.
\end{equation}
Hence the sequence $\left\{J_{Nm+p}\right\}_m$ converges uniformly by the Dirichlet test.
\end{proof}

The next result provides the initial condition for the partial sums.

\begin{lemma}
\label{lem:S_zero}
For any $N \in \mathbb{N}$, $p \in \mathbb{Z}$, and $n \in \mathbb{N}$ the following holds:
\begin{equation}
\label{eq:S_zero}
\sum_{m=0}^{n} \qty(\pm 1)^m J_{Nm+p}(0) = \begin{cases}
1,&\text{when $p=0$;}\\
0,&\text{otherwise.}
\end{cases}
\end{equation}
\end{lemma}

\begin{proof}
The Bessel function admits the power series expansion
\begin{equation}
J_\nu(x) = \sum_{k=0}^{\infty} \frac{(-1)^k}{k! \, \Gamma(\nu+k+1)} \left(\frac{x}{2}\right)^{\nu+2k}.
\end{equation}
At $x=0$, all terms with $\nu>0$ vanish due to the factor $(x/2)^{\nu}$. The only non-zero contribution occurs when $\nu=0$, giving $J_0(0)=1$. Therefore,
\begin{equation}
J_{Nm+p}(0) = \begin{cases}
1, & Nm+p = 0,\\
0, & \text{otherwise}.
\end{cases}
\end{equation}
Since $N \ge 1$ and $m \ge 0$, the condition $Nm+p=0$ holds iff $p=0$ and $m=0$.
Consequently, equation~(\ref{eq:S_zero}) is true.
\end{proof}

Having established the function series convergence and the initial values, we now derive the differential equation governing the partial sums.

\begin{lemma}
\label{lem:S_Cauchi}
Let us denote the partial sums of the series under consideration as
\begin{equation}
S^{(\pm, n)}_{N,p}(x) \equiv \sum_{m=0}^{n} \qty(\pm 1)^m J_{Nm+p}(x),
\end{equation}
and the vector $\vb{S^{(\pm, n)}_N}$ collects the partial sums for all $0 \le p < N$:
\begin{equation}
\vb{S^{(\pm,n)}_N}(x) \equiv \mqty[S^{(\pm, n)}_{N,0}(x) & S^{(\pm, n)}_{N,1}(x) & \cdots & S^{(\pm, n)}_{N,N-1}(x)].
\end{equation}
Then for any $N \in \mathbb{N}$, $n \in \mathbb{N}$, and $x \ge 0$ the partial sum vector $\vb{S^{(\pm, n)}_N}(x)$ satisfies the Cauchy problem for the following system of nonhomogeneous ordinary differential equations:
\begin{equation}
\dv{\vb{S^{(\pm, n)}_N}(x)}{x} = A^{(\pm)}_N \vb{S^{(\pm, n)}_N}(x) + \vb{f^{(\pm)}_N}(x) + \qty(\pm 1)^n \vb{f^{(\pm, n)}_N}(x),
\end{equation}
with initial conditions
\begin{equation}
\vb{S^{(\pm, n)}_N}(0) = \vb{e_0},
\end{equation}
where
\begin{equation}
\label{eq:A}
A^{(\pm)}_N \equiv \mqty[
0            & -\frac{1}{2} &              &        & \pm\frac{1}{2} \\
\frac{1}{2}  & 0            & -\frac{1}{2} &        & \\
             & \frac{1}{2}  & \ddots       & \ddots & \\
             &              & \ddots       &        & -\frac{1}{2}\\
\mp\frac{1}{2} &              &              & \frac{1}{2} & 0],
\end{equation}
\begin{equation}
\vb{f^{(\pm)}_N}(x) \equiv -\frac{1}{2} J_1(x) \vb{e_0} \pm \frac{1}{2} J_0(x) \vb{e_{N-1}},
\end{equation}
\begin{equation}
\vb{f^{(\pm, n)}_N}(x) \equiv \mp \frac{1}{2} J_{N(n+1)-1}(x) \vb{e_0} -\frac{1}{2} J_{N(n+1)}(x) \vb{e_{N-1}},
\end{equation}
\begin{equation}
\vb{e_0} \equiv \mqty[1 & 0 & \cdots & 0],
\end{equation}
\begin{equation}
\vb{e_{N-1}} \equiv \mqty[0 & 0 & \cdots & 1].
\end{equation}
\end{lemma}

\begin{proof}
Differentiating the definition of $S^{(\pm, n)}_{N,p}$ termwise and using the recurrence relation for Bessel functions,
\begin{equation}
\dv{J_\nu(x)}{x} = \frac{1}{2} J_{\nu-1}(x) - \frac{1}{2} J_{\nu+1}(x),
\end{equation}
we obtain
\begin{equation}
\dv{S^{(\pm, n)}_{N,p}(x)}{x} = \frac{1}{2} \sum_{m=0}^{n} \qty(\pm 1)^m \qty(J_{Nm+p-1}(x) - J_{Nm+p+1}(x)).
\end{equation}
For $0 < p < N-1$, shifting the summation index yields
\begin{equation}
\dv{S^{(\pm, n)}_{N,p}(x)}{x} = \frac{1}{2} S^{(\pm,n)}_{N,p-1}(x) - \frac{1}{2} S^{(\pm,n)}_{N,p+1}(x).
\end{equation}
The boundary cases $p=0$ and $p=N-1$ require special treatment because the indices $p-1$ and $p+1$ fall outside the range $0,\dots,N-1$. Using the relation $J_{-\nu}(x) = (-1)^\nu J_\nu(x)$, we obtain for $p=0$:
\begin{equation}
\dv{S^{(\pm,n)}_{N,0}(x)}{x} = -\frac{1}{2} S^{(\pm,n)}_{N,1}(x) \pm \frac{1}{2} S^{(\pm,n)}_{N,N-1}(x)
- \frac{1}{2} J_{1}(x) - \frac{\qty(\pm 1)}{2}^{n+1} J_{N(n+1)-1}(x),
\end{equation}
and for $p=N-1$:
\begin{equation}
\dv{S^{(\pm,n)}_{N,N-1}(x)}{x} = \frac{1}{2} S^{(\pm,n)}_{N,N-2}(x) \mp \frac{1}{2} S^{(\pm,n)}_{N,0}(x)
\pm \frac{1}{2} J_{0}(x) - \frac{\qty(\pm 1)}{2}^{n} J_{N(n+1)}(x).
\end{equation}
Collecting these equations into vector form gives the nonhomogeneous linear system
\begin{equation}
\dv{\vb{S^{(\pm,n)}_N}(x)}{x} = A^{(\pm)}_N \vb{S^{(\pm,n)}_N}(x) + \vb{f^{(\pm)}_N}(x) + \qty(\pm 1)^n \vb{f^{(\pm,n)}_N}(x),
\end{equation}
with $A^{(\pm)}_N$, $\vb{f^{(\pm)}_N}$, and $\vb{f^{(\pm,n)}_N}$ as defined in the statement. The initial condition follows directly from Lemma~\ref{lem:S_zero}.
\end{proof}

To clarify the notation of Lemma~\ref{lem:S_Cauchi}, consider the following special cases.
For $N=1$, $A^{(\pm)}_1=0,$
$\vb{f^{(\pm)}_1}(x) = \mqty[-\frac{1}{2} J_1(x) \pm \frac{1}{2} J_0(x)],$
and
$\vb{f^{(\pm,n)}_1}(x) = \mqty[\mp \frac{1}{2} J_{n}(x) -\frac{1}{2} J_{n+1}(x)]$.
For $N=2$, $A^{(+)}_2=0$ and
$$
A^{(-)}_2 = \mqty[0 & -1\\1 & 0].
$$

To solve the system from Lemma~\ref{lem:S_Cauchi}, the matrix exponential $\exp(x A^{(\pm)}_N)$ must be evaluated.
This requires knowledge of the eigenstructure of $A^{(\pm)}_N$.

\begin{lemma}
\label{lem:A_eigen}
Matrix~$A^{(\pm)}_N$ in equation~(\ref{eq:A}) has eigenvalues
\begin{equation}
\lambda^{(+)}_k = -i \sin(\frac{2\pi}{N}k),
\end{equation}
\begin{equation}
\lambda^{(-)}_k = -i \sin(\frac{2\pi}{N}\qty(k+\frac{1}{2})).
\end{equation}
with corresponding eigenvectors
\begin{equation}
\vb{v^{(+)}_k} = \frac{1}{\sqrt{N}} \mqty[1& \omega^k& \omega^{2k}& \cdots& \omega^{\qty(N-1)k}],\\
\end{equation}
\begin{equation}
\vb{v^{(-)}_k} = \frac{1}{\sqrt{N}} \mqty[1& \omega^{k+\frac{1}{2}}& \omega^{2\qty(k+\frac{1}{2})}& \cdots& \omega^{\qty(N-1)\qty(k+\frac{1}{2})}].
\end{equation}
where $\omega \equiv \exp(\frac{2\pi i}{N})$, $i$ denotes the imaginary unit, $0 \le k < N$.
The multiplicity of all eigenvalues is one.
\end{lemma}

\begin{proof}
$A^{(+)}_N$ is a circulant matrix, each of whose rows is a cyclic shift of the row above to the right.
$A^{(-)}_N$ is a skew-circulant matrix, whose elements change sign upon a cyclic shift through the rightmost position.
Eigenvalues and eigenvectors for both kinds of matrices are well known; see, for instance,~\cite{Karner2003}.
\end{proof}

Note that $A^{(\pm)}_N$ is a real matrix, so each eigenvalue $\lambda^{(\pm)}_k$ is either zero or part of a complex conjugate pair.

With the eigendecomposition of $A^{(\pm)}_N$ at hand, the main result can now be stated.

\begin{theorem}
\label{theorem:series}
For any $N \in \mathbb{N}$, $0 \le p < N$,
\begin{multline}
\label{eq:series} 
S^{(\pm)}_{N,p}(x) \equiv \sum_{m=0}^{\infty} \qty(\pm 1)^m J_{Nm+p}(x) = \\
= \frac{\delta_{p0}}{2} J_0(x) + \frac{1}{2N} \sum_{k=0}^{N-1} \left(\cos(\frac{2\pi}{N} p c^{(\pm)}_k - x a^{(\pm)}_k) + \right.\\
\left.\cos(\frac{2\pi}{N} c^{(\pm)}_k) \qty(\JC0(a^{(\pm)}_k, x) \cos(\frac{2\pi}{N} p c^{(\pm)}_k - x a^{(\pm)}_k) + \JS0(a^{(\pm)}_k, x) \sin(x a^{(\pm)}_k - \frac{2\pi}{N} p c^{(\pm)}_k))\right),
\end{multline}
where $\delta_{ij}$ is a Kronecker delta symbol, $c_k^{(+)} = k,$ $c_k^{(-)} = k+\frac{1}{2},$
and $a^{(\pm)}_k \equiv \sin(\frac{2\pi}{N}c^{(\pm)}_k)$. $\JC0(a, x)$ and $\JS0(a, x)$ denote incomplete Lipschitz-Hankel integrals of Bessel functions:
\begin{equation}
\label{eq:jc0}
\JC0(a, x) = \int_0^x \cos(a t) J_0(t) dt,
\end{equation}
\begin{equation}
\label{eq:js0}
\JS0(a, x) = \int_0^x \sin(a t) J_0(t) dt.
\end{equation}
\end{theorem}

\begin{proof}
The solution of the Cauchy problem from Lemma~\ref{lem:S_Cauchi} is given by the variation-of-constants formula:
\begin{equation}
\vb{S^{(\pm, n)}_N}(x) = \exp\qty(x A^{(\pm)}_N) \vb{S^{(\pm,n)}_N}(0) + \int_0^{x} dt \exp\qty((x-t) A^{(\pm)}_N) \qty(\vb{f^{(\pm)}_N}(t) + \qty(\pm 1)^n \vb{f^{(\pm,n)}_N}(t)),
\end{equation}
where $\exp\qty(x A^{(\pm)}_N)$ denotes the matrix exponential.
By Lemma~\ref{lem:J_conv}, $J_{N(n+1)}(x) \to 0$ and $J_{N(n+1)-1}(x) \to 0$ uniformly as $n \to \infty$, hence $\vb{f^{(\pm,n)}_N}(x) \to 0$ uniformly. Taking the limit $n \to \infty$ and denoting $\vb{S^{(\pm)}_N}(x) \equiv \lim_{n\to\infty} \vb{S^{(\pm,n)}_N}(x)$, we obtain
\begin{equation}
\label{eq:S_N}
\vb{S^{(\pm)}_N}(x) = \exp\qty(x A^{(\pm)}_N) \vb{S^{(\pm)}_N}(0) + \int_0^{x} dt \exp\qty((x-t) A^{(\pm)}_N) \vb{f^{(\pm)}_N}(t),
\end{equation}
The matrix exponential $\exp\qty(x A^{(\pm)}_N)$ is evaluated via the eigendecomposition from Lemma~\ref{lem:A_eigen}.
Let $V^{(\pm)}$ be the matrix whose columns are the eigenvectors $\vb{v^{(\pm)}_k}$ and $\Lambda^{(\pm)} = \operatorname{diag}(\lambda^{(\pm)}_0,\dots,\lambda^{(\pm)}_{N-1})$.
Then $A^{(\pm)}_N = V^{(\pm)} \Lambda^{(\pm)} V^{(\pm)H}$ and
\begin{equation}
\exp\qty(x A^{(\pm)}_N) = V^{(\pm)} \exp\qty(x \Lambda^{(\pm)}) V^{(\pm)H},
\end{equation}
where $V^{(\pm)H}$ denotes the Hermitian adjoint of $V^{(\pm)}$.

Then
\begin{equation}
\qty(V^{(\pm)} \exp\qty(x \Lambda^{(\pm)}) V^{(\pm)H} \vb{e_0})_p = \frac{1}{N} \sum_{k=0}^{N-1} \exp(\frac{2\pi i}{N} p c^{(\pm)}_k - i a^{(\pm)}_k x),
\end{equation}
\begin{equation}
\qty(V^{(\pm)} \exp\qty(x \Lambda^{(\pm)}) V^{(\pm)H} \vb{e_{N-1}})_p = \pm \frac{1}{N} \sum_{k=0}^{N-1} \exp(\frac{2\pi i}{N} p c^{(\pm)}_k - i a^{(\pm)}_k x + \frac{2\pi i}{N} c^{(\pm)}_k),
\end{equation}
where $c^{(\pm)}_k$ and $a^{(\pm)}_k$ were defined in the statement of Theorem~\ref{theorem:series}.

The following identity will also be needed:
\begin{equation}
\int_0^x \exp(i a t) J_1(t) dt = 1 - \exp(iax) J_0(x) + i a \int_0^x \exp(i a t) J_0(t) dt.
\end{equation}

Equation~(\ref{eq:S_N}) is then rewritten into the form of equation~(\ref{eq:series}).
Since $A^{(\pm)}_N$ is a real matrix, only the real part of each complex exponential contributes, as noted in Lemma~\ref{lem:A_eigen}.

\end{proof}

Note that the incomplete Lipschitz-Hankel integrals of Bessel functions $\JC0(a, x)$ and $\JS0(a, x)$ generally cannot be reduced to more familiar functions and are therefore treated as a separate class of special functions~\cite{Agrest1971}.
It has also been demonstrated that incomplete Lipschitz-Hankel integrals of Bessel functions can be expressed in terms of incomplete Bessel and Struve functions~\cite{Agrest1971}.

\begin{lemma}
\label{lem:ilhi}
\begin{equation}
\JC0(-a, x) = \JC0(a, x),
\end{equation}
\begin{equation}
\label{eq:jc0_mx}
\JC0(a, -x) = -\JC0(a, x),
\end{equation}
\begin{equation}
\JS0(-a, x) = -\JS0(a, x),
\end{equation}
\begin{equation}
\label{eq:js0_mx}
\JS0(a, -x) = \JS0(a, x).
\end{equation}
\end{lemma}
\begin{proof}
The result follows from the definitions~(\ref{eq:jc0})--(\ref{eq:js0}) and the symmetry properties of the sine, cosine, and $J_0$ functions.
\end{proof}

For reference, we provide explicit results for $N$ up to $6$ in Table~\ref{tab:neumann_1} for $\sum_{m=0}^{\infty} J_{Nm+p}(x)$
and in Table~\ref{tab:neumann_2} for $\sum_{m=0}^{\infty} \qty(- 1)^m J_{Nm+p}(x)$.
Lemma~\ref{lem:ilhi} has been used to simplify the final expressions.

\begingroup
\renewcommand{\arraystretch}{2.25}

\begin{longtable}{|c|c|L|}
    \caption{Multisectioned series of Bessel functions.} \label{tab:neumann_1} \\

    \hline
    $N$ & $p$ & S_{N,p} \\ \hline
    \endfirsthead

    \hline
    \multicolumn{3}{|c|}{Continued from previous page} \\ \hline
    $N$ & $p$ & S_{N,p} \\ \hline
    \endhead

    \hline
    \multicolumn{3}{|r|}{Continued on next page...} \\ \hline
    \endfoot

    \hline
    \endlastfoot

    $1$ & $0$ & \frac{1}{2} \qty(\JC0(0,x)+J_0(x)+1) \hfill\text{(5.7.1)~.1 in \cite{Prudnikov1986}} \\

    $2$ & $0$ & \frac{1}{2} \qty(J_0(x)+1) \hfill\text{(8.512~1) in \cite{Gradshteyn2014}, (9.1.46) in \cite{Abramowitz1965}, (5.7.1)~.14 in \cite{Prudnikov1986}}\\
    $2$ & $1$ & \frac{1}{2} \JC0(0,x) \hfill\text{(8.514~7) in \cite{Gradshteyn2014}, (5.7.1)~.15 in \cite{Prudnikov1986}} \\

    $3$ & $0$ & \frac{1}{6} \qty(\JC0(0,x)-\cos(\frac{\sqrt{3} x}{2}) \qty(\JC0\qty(\frac{\sqrt{3}}{2},x)-2)-\sin(\frac{\sqrt{3}x}{2}) \JS0\qty(\frac{\sqrt{3}}{2},x)+3 J_0(x)+1)\\
    $3$ & $1$ & \frac{1}{6} \qty(\JC0(0,x)+\sin(\frac{\pi}{6} - \frac{\sqrt{3}x}{2}) \qty(\JC0\qty(\frac{\sqrt{3}}{2},x)-2)+\cos(\frac{\pi}{6} -\frac{\sqrt{3} x}{2}) \JS0\qty(\frac{\sqrt{3}}{2},x)+1) \\
    $3$ & $2$ & \frac{1}{6} \qty(\JC0(0,x)+\sin(\frac{\sqrt{3}x}{2}+\frac{\pi}{6}) \qty(\JC0\qty(\frac{\sqrt{3}}{2},x)-2)-\cos(\frac{\sqrt{3} x}{2}+\frac{\pi}{6}) \JS0\qty(\frac{\sqrt{3}}{2},x)+1) \\

    $4$ & $0$ & \frac{1}{4} \qty(\cos (x)+2 J_0(x)+1) \hfill\text{(5.7.1)~.19 in \cite{Prudnikov1986}}\\
    $4$ & $1$ & \frac{1}{4} \qty(\JC0(0,x)+\sin(x)) \\
    $4$ & $2$ & \frac{1}{4} \qty(1-\cos (x)) \hfill\text{(5.7.1)~.20 in \cite{Prudnikov1986}} \\
    $4$ & $3$ & \frac{1}{4} \qty(\JC0(0,x)-\sin(x)) \\

    $5$ & $0$ & \begin{aligned}
\frac{1}{40} & \left(\qty(\sqrt{5}-1) \sin(x\sin(\frac{2\pi}{5})) \qty(\JC0\qty(\sin(\frac{2\pi}{5}),x)+\JS0\qty(\sin(\frac{2\pi}{5}),x))\right.\\
&+ \left.\qty(1+\sqrt{5}) \left(-\sin(x\sin(\frac{\pi}{5})) \qty(\JC0\qty(\sin(\frac{\pi}{5}),x)+\JS0\qty(\sin(\frac{\pi}{5}),x))\right.\right.\\
&\qquad- \left.\left.2\cos(x\sin(\frac{\pi}{5})) \JC0\qty(\sin(\frac{\pi}{5}),x)\right)\right.\\
&+ \left.2\cos(x\sin(\frac{2\pi}{5})) \qty(\qty(\sqrt{5}-1) \JC0\qty(\sin(\frac{2\pi}{5}),x)+4)\right.\\
&+ \left.4\JC0(0,x)+8 \cos(x\sin(\frac{\pi}{5}))+20 J_0(x)+4\right)
\end{aligned}\\
    $5$ & $1$ & \begin{aligned}
\frac{1}{40} & \left(\qty(1+\sqrt{5}) \left(\sin(x\sin(\frac{\pi}{5})+\frac{\pi}{5}) \qty(\JC0\qty(\sin(\frac{\pi}{5}),x)+\JS0\qty(\sin(\frac{\pi}{5}),x))\right.\right.\\
&\qquad+ \left.\left.2\cos(x\sin(\frac{\pi}{5})+\frac{\pi}{5}) \JC0\qty(\sin(\frac{\pi}{5}),x)\right)\right.\\
&+ \left.\qty(\sqrt{5}-1) \left(2 \sin(x\sin(\frac{2\pi}{5})+\frac{\pi}{10}) \JC0\qty(\sin(\frac{2\pi}{5}),x)\right.\right.\\
&\qquad- \left.\left.\cos(x\sin(\frac{2\pi}{5})+\frac{\pi}{10}) \qty(\JC0\qty(\sin(\frac{2\pi}{5}),x)+\JS0\qty(\sin(\frac{2\pi}{5}),x))\right)\right.\\
&+ \left.4\JC0(0,x)+8 \sin(x\sin(\frac{2\pi}{5})+\frac{\pi}{10})-8 \cos(x\sin(\frac{\pi}{5})+\frac{\pi}{5})+4\right)
\end{aligned}\\
    $5$ & $2$ &\begin{aligned}
\frac{1}{40} & \left(-\qty(\sqrt{5}-1) \sin(x\sin(\frac{2\pi}{5})+\frac{\pi}{5}) \qty(\JC0\qty(\sin(\frac{2\pi}{5}),x)+\JS0\qty(\sin(\frac{2\pi}{5}),x))\right.\\
&- \left.\qty(1+\sqrt{5}) \left(-\cos(\frac{\pi}{10}-x\sin(\frac{\pi}{5})) \qty(\JC0\qty(\sin(\frac{\pi}{5}),x)+\JS0\qty(\sin(\frac{\pi}{5}),x))\right.\right.\\
&\qquad- \left.\left.2\sin(\frac{\pi}{10}-x\sin(\frac{\pi}{5})) \JC0\qty(\sin(\frac{\pi}{5}),x)\right)\right.\\
&+ \left.2\cos(x\sin(\frac{2\pi}{5})+\frac{\pi}{5}) \qty(-\qty(\qty(\sqrt{5}-1) \JC0\qty(\sin(\frac{2\pi}{5}),x))-4)\right.\\
&+ \left.4\JC0(0,x)+8 \sin(\frac{\pi}{10}-x\sin(\frac{\pi}{5}))+4\right)
\end{aligned}\\
    $5$ & $3$ &\begin{aligned}
\frac{1}{40} & \left(\qty(1+\sqrt{5}) \left(\cos(x\sin(\frac{\pi}{5})+\frac{\pi}{10}) \qty(\JC0\qty(\sin(\frac{\pi}{5}),x)+\JS0\qty(\sin(\frac{\pi}{5}),x))\right.\right.\\
&\qquad- \left.\left.2\sin(x\sin(\frac{\pi}{5})+\frac{\pi}{10}) \JC0\qty(\sin(\frac{\pi}{5}),x)\right)\right.\\
&+ \left.\qty(\sqrt{5}-1) \left(\sin(\frac{\pi}{5}-x\sin(\frac{2\pi}{5})) \qty(\JC0\qty(\sin(\frac{2\pi}{5}),x)+\JS0\qty(\sin(\frac{2\pi}{5}),x))\right.\right.\\
&\qquad- \left.\left.2\cos(\frac{\pi}{5}-x\sin(\frac{2\pi}{5})) \JC0\qty(\sin(\frac{2\pi}{5}),x)\right)\right.\\
&+ \left.4\JC0(0,x)+8 \sin(x\sin(\frac{\pi}{5})+\frac{\pi}{10})-8\cos(\frac{\pi}{5}-x\sin(\frac{2\pi}{5}))+4\right)
\end{aligned}\\
    $5$ & $4$ &\begin{aligned}
\frac{1}{40} & \left(\qty(1+\sqrt{5}) \left(2 \cos(\frac{\pi}{5}-x\sin(\frac{\pi}{5})) \JC0\qty(\sin(\frac{\pi}{5}),x)\right.\right.\\
&\qquad- \left.\left.\sin(\frac{\pi}{5}-x\sin(\frac{\pi}{5})) \qty(\JC0\qty(\sin(\frac{\pi}{5}),x)+\JS0\qty(\sin(\frac{\pi}{5}),x))\right)\right.\\
&+ \left.\qty(\sqrt{5}-1) \left(\cos(\frac{\pi}{10}-x\sin(\frac{2\pi}{5})) \qty(\JC0\qty(\sin(\frac{2\pi}{5}),x)+\JS0\qty(\sin(\frac{2\pi}{5}),x))\right.\right.\\
&\qquad+ \left.\left.2\sin(\frac{\pi}{10}-x\sin(\frac{2\pi}{5})) \JC0\qty(\sin(\frac{2\pi}{5}),x)\right)\right.\\
&+ \left.4\JC0(0,x)+8 \sin(\frac{\pi}{10}-x\sin(\frac{2\pi}{5}))-8\cos(\frac{\pi}{5}-x\sin(\frac{\pi}{5}))+4\right)
\end{aligned}\\

    $6$ & $0$ & \frac{1}{6} \qty(2 \cos(\frac{\sqrt{3} x}{2})+3 J_0(x)+1)\\
    $6$ & $1$ & \frac{1}{12} \qty(2 \JC0(0,x)+\cos(\frac{\sqrt{3} x}{2}) \JC0\qty(\frac{\sqrt{3}}{2},x)+\sin(\frac{\sqrt{3}x}{2}) \qty(\JS0\qty(\frac{\sqrt{3}}{2},x)+2 \sqrt{3})) \\
    $6$ & $2$ & \frac{1}{12} \qty(\sqrt{3}\sin(\frac{\sqrt{3}x}{2}) \JC0\qty(\frac{\sqrt{3}}{2},x)-\cos(\frac{\sqrt{3} x}{2}) \qty(\sqrt{3} \JS0\qty(\frac{\sqrt{3}}{2},x)+2)+2) \\
    $6$ & $3$ & \frac{1}{6} \qty(\JC0(0,x)-\cos(\frac{\sqrt{3} x}{2}) \JC0\qty(\frac{\sqrt{3}}{2},x)-\sin(\frac{\sqrt{3}x}{2}) \JS0\qty(\frac{\sqrt{3}}{2},x)) \\
    $6$ & $4$ & \frac{1}{12} \qty(-\sqrt{3} \sin(\frac{\sqrt{3}x}{2}) \JC0\qty(\frac{\sqrt{3}}{2},x)+\cos(\frac{\sqrt{3} x}{2}) \qty(\sqrt{3} \JS0\qty(\frac{\sqrt{3}}{2},x)-2)+2) \\
    $6$ & $5$ & \frac{1}{12} \qty(2 \JC0(0,x)+\cos(\frac{\sqrt{3} x}{2}) \JC0\qty(\frac{\sqrt{3}}{2},x)+\sin(\frac{\sqrt{3}x}{2}) \qty(\JS0\qty(\frac{\sqrt{3}}{2},x)-2 \sqrt{3}))

\end{longtable}

\begin{longtable}{|c|c|L|}
    \caption{Alternating multisectioned series of Bessel functions} \label{tab:neumann_2} \\
    
    \hline
    $N$ & $p$ & S_{N,p} \\ \hline
    \endfirsthead

    \hline
    \multicolumn{3}{|c|}{Continued from previous page} \\ \hline
    $N$ & $p$ & S_{N,p} \\ \hline
    \endhead

    \hline
    \multicolumn{3}{|r|}{Continued on next page...} \\ \hline
    \endfoot

    \hline
    \endlastfoot

    $1$ & $0$ & \frac{1}{2} \qty(-\JC0(0,x)+J_0(x)+1) \hfill\text{(5.7.1)~.1 in \cite{Prudnikov1986}}\\

    $2$ & $0$ & \frac{1}{2} \qty(\cos (x)+J_0(x)) \hfill\text{(8.514 2) in \cite{Gradshteyn2014}, (9.1.47) in \cite{Abramowitz1965}, (5.7.1)~.14 in \cite{Prudnikov1986}}\\
    $2$ & $1$ & \frac{1}{2} \sin (x) \hfill\text{(8.514 1) in \cite{Gradshteyn2014}, (9.1.48) in \cite{Abramowitz1965}, (5.7.1)~.15 in \cite{Prudnikov1986}}\\

    $3$ & $0$ & \frac{1}{6} \qty(-\JC0(0,x)+\cos(\frac{\sqrt{3}x}{2}) \qty(\JC0\qty(\frac{\sqrt{3}}{2},x)+2)+\sin(\frac{\sqrt{3}x}{2}) \JS0\qty(\frac{\sqrt{3}}{2},x)+3 J_0(x)+1) \\
    $3$ & $1$ & \frac{1}{6} \qty(\JC0(0,x)+\sin(\frac{\sqrt{3}x}{2}+\frac{\pi}{6}) \qty(\JC0\qty(\frac{\sqrt{3}}{2},x)+2)-\cos(\frac{\sqrt{3}x}{2}+\frac{\pi}{6}) \JS0\qty(\frac{\sqrt{3}}{2},x)-1) \\
    $3$ & $2$ & \frac{1}{6} \qty(-\JC0(0,x)-\sin(\frac{\pi}{6}-\frac{\sqrt{3}x}{2}) \qty(\JC0\qty(\frac{\sqrt{3}}{2},x)+2)-\cos(\frac{\pi}{6}-\frac{\sqrt{3}x}{2}) \JS0\qty(\frac{\sqrt{3}}{2},x)+1) \\

    $4$ & $0$ & \frac{1}{2} \qty(\cos(\frac{x}{\sqrt{2}})+J_0(x)) \hfill\text{(5.7.1)~.19 in \cite{Prudnikov1986}}\\
    $4$ & $1$ & \frac{1}{4} \qty(\cos(\frac{x}{\sqrt{2}}) \JC0\qty(\frac{1}{\sqrt{2}},x)+\sin(\frac{x}{\sqrt{2}}) \qty(\JS0\qty(\frac{1}{\sqrt{2}},x)+\sqrt{2})) \\
    $4$ & $2$ & \frac{1}{2\sqrt{2}}\qty(\sin(\frac{x}{\sqrt{2}}) \JC0\qty(\frac{1}{\sqrt{2}},x)-\cos(\frac{x}{\sqrt{2}}) \JS0\qty(\frac{1}{\sqrt{2}},x)) \\
    $4$ & $3$ & \frac{1}{4} \qty(\sin(\frac{x}{\sqrt{2}}) \qty(\sqrt{2}-\JS0\qty(\frac{1}{\sqrt{2}},x))-\cos(\frac{x}{\sqrt{2}}) \JC0\qty(\frac{1}{\sqrt{2}},x)) \\

    $5$ & $0$ & \begin{aligned}
\frac{1}{40} & \left(-\qty(\sqrt{5}-1) \sin(x\sin(\frac{2\pi}{5})) \qty(\JC0\qty(\sin(\frac{2\pi}{5}),x)+\JS0\qty(\sin(\frac{2\pi}{5}),x))\right.\\
& +\left.\qty(1+\sqrt{5}) \left(\sin(x\sin(\frac{\pi}{5})) \qty(\JC0\qty(\sin(\frac{\pi}{5}),x)+\JS0\qty(\sin(\frac{\pi}{5}),x))\right.\right.\\
& \qquad\left.\left.+2\cos(x\sin(\frac{\pi}{5})) \JC0\qty(\sin(\frac{\pi}{5}),x)\right)\right.\\
& +2\left.\cos(x\sin(\frac{2\pi}{5})) \qty(4-\qty(\sqrt{5}-1) \JC0\qty(\sin(\frac{2\pi}{5}),x))\right.\\
& \left.-4 \JC0(0,x)+8 \cos(x\sin(\frac{\pi}{5}))+20 J_0(x)+4\right)
\end{aligned}\\
    $5$ & $1$ & \begin{aligned}
\frac{1}{40} & \left(\qty(1+\sqrt{5}) \left(2 \cos(\frac{\pi}{5}-x\sin(\frac{\pi}{5})) \JC0\qty(\sin(\frac{\pi}{5}),x)\right.\right.\\
& \qquad\left.\left.-\sin(\frac{\pi}{5}-x\sin(\frac{\pi}{5})) \qty(\JC0\qty(\sin(\frac{\pi}{5}),x)+\JS0\qty(\sin(\frac{\pi}{5}),x))\right)\right.\\
& +\left.\qty(\sqrt{5}-1) \left(\cos(\frac{\pi}{10}-x\sin(\frac{2\pi}{5})) \qty(\JC0\qty(\sin(\frac{2\pi}{5}),x)+\JS0\qty(\sin(\frac{2\pi}{5}),x))\right.\right.\\
& \qquad\left.\left.+2\sin(\frac{\pi}{10}-x\sin(\frac{2\pi}{5})) \JC0\qty(\sin(\frac{2\pi}{5}),x)\right)\right.\\
& +\left.4\JC0(0,x)-8 \sin(\frac{\pi}{10}-x\sin(\frac{2\pi}{5}))+8 \cos(\frac{\pi}{5}-x\sin(\frac{\pi}{5}))-4\right)
\end{aligned}\\
    $5$ & $2$ & \begin{aligned}
\frac{1}{40} & \left(\qty(1+\sqrt{5}) \left(2 \sin(x\sin(\frac{\pi}{5})+\frac{\pi}{10}) \JC0\qty(\sin(\frac{\pi}{5}),x)\right.\right.\\
& \qquad\left.\left.-\cos(x\sin(\frac{\pi}{5})+\frac{\pi}{10}) \qty(\JC0\qty(\sin(\frac{\pi}{5}),x)+\JS0\qty(\sin(\frac{\pi}{5}),x))\right)\right.\\
& \left.+\qty(\sqrt{5}-1) \left(2 \cos(\frac{\pi}{5}-x\sin(\frac{2\pi}{5})) \JC0\qty(\sin(\frac{2\pi}{5}),x)\right.\right.\\
& \qquad\left.\left.-\sin(\frac{\pi}{5}-x\sin(\frac{2\pi}{5})) \qty(\JC0\qty(\sin(\frac{2\pi}{5}),x)+\JS0\qty(\sin(\frac{2\pi}{5}),x))\right)\right.\\
& \left.-4 \JC0(0,x)+8 \sin(x\sin(\frac{\pi}{5})+\frac{\pi}{10})-8 \cos(\frac{\pi}{5}-x\sin(\frac{2\pi}{5}))+4\right)
\end{aligned}\\
    $5$ & $3$ & \begin{aligned}
\frac{1}{40} & \left(-\qty(\sqrt{5}-1) \sin(x\sin(\frac{2\pi}{5})+\frac{\pi}{5}) \qty(\JC0\qty(\sin(\frac{2\pi}{5}),x)+\JS0\qty(\sin(\frac{2\pi}{5}),x))\right.\\
& +\left.\qty(1+\sqrt{5}) \left(-\cos(\frac{\pi}{10}-x\sin(\frac{\pi}{5})) \qty(\JC0\qty(\sin(\frac{\pi}{5}),x)+\JS0\qty(\sin(\frac{\pi}{5}),x))\right.\right.\\
& \qquad\left.\left.-2\sin(\frac{\pi}{10}-x\sin(\frac{\pi}{5})) \JC0\qty(\sin(\frac{\pi}{5}),x)\right)\right.\\
& +\left.2\cos(x\sin(\frac{2\pi}{5})+\frac{\pi}{5}) \qty(4-\qty(\sqrt{5}-1) \JC0\qty(\sin(\frac{2\pi}{5}),x))\right.\\
& +\left.4\JC0(0,x)-8 \sin(\frac{\pi}{10}-x\sin(\frac{\pi}{5}))-4\right)
\end{aligned}\\
    $5$ & $4$ & \begin{aligned}
\frac{1}{40} & \left(\qty(1+\sqrt{5}) \left(-\sin(x\sin(\frac{\pi}{5})+\frac{\pi}{5}) \qty(\JC0\qty(\sin(\frac{\pi}{5}),x)+\JS0\qty(\sin(\frac{\pi}{5}),x))\right.\right.\\
& \qquad\left.\left.-2\cos(x\sin(\frac{\pi}{5})+\frac{\pi}{5}) \JC0\qty(\sin(\frac{\pi}{5}),x)\right)\right.\\
& \left.+\qty(\sqrt{5}-1) \left(\cos(x\sin(\frac{2\pi}{5})+\frac{\pi}{10}) \qty(\JC0\qty(\sin(\frac{2\pi}{5}),x)+\JS0\qty(\sin(\frac{2\pi}{5}),x))\right.\right.\\
& \qquad\left.\left.-2\sin(x\sin(\frac{2\pi}{5})+\frac{\pi}{10}) \JC0\qty(\sin(\frac{2\pi}{5}),x)\right)\right.\\
& -\left.4\JC0(0,x)+8 \sin(x\sin(\frac{2\pi}{5})+\frac{\pi}{10})-8 \cos(x\sin(\frac{\pi}{5})+\frac{\pi}{5})+4\right)
\end{aligned}\\

    $6$ & $0$ & \frac{1}{6} \qty(2 \cos(\frac{x}{2})+\cos (x)+3 J_0(x)) \\
    $6$ & $1$ & \frac{1}{12} \qty(3 \cos(\frac{x}{2}) \JC0\qty(\frac{1}{2},x)+\sin(\frac{x}{2}) \qty(3 \JS0\qty(\frac{1}{2},x)+2)+2 \sin (x)) \\
    $6$ & $2$ & \frac{1}{12} \qty(3 \sin(\frac{x}{2}) \JC0\qty(\frac{1}{2},x)+\cos(\frac{x}{2}) \qty(2-3 \JS0\qty(\frac{1}{2},x))-2 \cos (x)) \\
    $6$ & $3$ & \frac{4}{3} \sin[3](\frac{x}{4}) \cos(\frac{x}{4}) \\
    $6$ & $4$ & \frac{1}{12} \qty(3 \sin(\frac{x}{2}) \JC0\qty(\frac{1}{2},x)-\cos(\frac{x}{2}) \qty(3 \JS0\qty(\frac{1}{2},x)+2)+2 \cos (x)) \\
    $6$ & $5$ & \frac{1}{12} \qty(-3 \cos(\frac{x}{2}) \JC0\qty(\frac{1}{2},x)+\sin(\frac{x}{2}) \qty(2-3 \JS0\qty(\frac{1}{2},x))+2 \sin (x))
\end{longtable}

\endgroup

\section{Validation and evaluation}
\label{sec:numerical_evaluation}

We begin this section with the following useful lemma.
\begin{lemma}[Alternative form for the solution]
\label{lem:series_alt}
For any $N \in \mathbb{N}$, $0 \le p < N$,
\begin{multline}
\label{eq:series_alt}
S^{(\pm)}_{N,p}(x) \equiv \sum_{m=0}^{\infty} \qty(\pm 1)^m J_{Nm+p}(x) = \\
= \frac{\delta_{p0}}{2} J_0(x) + \frac{\qty(-1)^p}{2N} \sum_{k=0}^{N-1} \left(\cos(\frac{2\pi}{N} p c^{(\sigma)}_k + x a^{(\sigma)}_k)  \right.\\
\left.-\cos(\frac{2\pi}{N} c^{(\sigma)}_k) \qty(\JC0(a^{(\sigma)}_k, x) \cos(\frac{2\pi}{N} p c^{(\sigma)}_k + x a^{(\sigma)}_k) + \JS0(a^{(\sigma)}_k, x) \sin(x a^{(\sigma)}_k + \frac{2\pi}{N} p c^{(\sigma)}_k))\right),
\end{multline}
where $c^{(\sigma)}_k$ denotes $c^{(\pm)}_k$ for even $N$ and $c^{(\mp)}_k$ for odd $N$,
while $a^{(\sigma)}_k \equiv \sin(\frac{2\pi}{N}c^{(\sigma)}_k)$.
All remaining notation is as in Theorem~\ref{theorem:series}.
\end{lemma}

\begin{proof}
First, note the following properties for $a^{(\pm)}_k$, which follow from the definition and properties of the sine function.
When $N$ is even, then
\begin{equation}
\label{eq:a_even_plus}
a^{(+)}_k = a^{(+)}_{N/2 - k} = a^{(+)}_{3N/2 - k},
\end{equation}
and
\begin{equation}
\label{eq:a_even_minus}
a^{(-)}_k = a^{(-)}_{N/2 - 1 - k} = a^{(-)}_{3N/2 -1 - k}.
\end{equation}
When $N$ is odd, then
\begin{equation}
\label{eq:a_odd}
a^{(\pm)}_k = a^{(\mp)}_{\lfloor N/2 \rfloor - k} = a^{(\mp)}_{N + \lfloor N/2 \rfloor - k},
\end{equation}
where $\lfloor \cdot \rfloor$ denotes the floor operation.

Second, we reorder the summation in equation~(\ref{eq:series}) according to the following pattern.
When $N$ is even, then
\begin{equation}
\sum_{k=0}^{N-1} f^{(+)}_k\qty(a_k, x) = \sum_{j=0}^{N/2} f^{(+)}_{N/2-j}\qty(a_j, x) + \sum_{j=N/2+1}^{N-1} f^{(+)}_{3N/2-j}\qty(a_j, x),
\end{equation}
\begin{equation}
\sum_{k=0}^{N-1} f^{(-)}_k\qty(a_k, x) = \sum_{j=0}^{N/2-1} f^{(-)}_{N/2-1-j}\qty(a_j, x) + \sum_{j=N/2}^{N-1} f^{(-)}_{3N/2-1-j}\qty(a_j, x).
\end{equation}
When $N$ is odd, then
\begin{equation}
\sum_{k=0}^{N-1} f^{(\pm)}_k\qty(a_k, x) = \sum_{j=0}^{\left\lfloor N/2 \right\rfloor-1} f^{(\mp)}_{\left\lfloor N/2 \right\rfloor-j}\qty(a_j, x) + \sum_{j=\left\lfloor N/2 \right\rfloor}^{N-1} f^{(\mp)}_{N + \left\lfloor N/2 \right\rfloor-j}\qty(a_j, x).
\end{equation}
where $f^{(\pm)}$ denotes the summand in equation~(\ref{eq:series}) and properties~(\ref{eq:a_even_plus})--(\ref{eq:a_odd}) have been applied.
Applying the periodicity of the cosine and sine functions yields the result stated in the lemma.
\end{proof}

\subsection{Symmetry of the result}
A useful check of equation~(\ref{eq:series}) is to verify its symmetry.
By the parity property of the Bessel function $J_{\nu}$, the following identity must hold:
\begin{equation}
\sum_{m=0}^{\infty} \qty(\pm 1)^m J_{Nm+p}(-x) = \begin{cases}
\displaystyle (-1)^p \sum_{m=0}^{\infty} \qty(\pm 1)^m J_{Nm+p}(x), & \text{when $N$ is even;}\\
\displaystyle (-1)^p \sum_{m=0}^{\infty} \qty(\mp 1)^m J_{Nm+p}(x), & \text{when $N$ is odd.}\\
\end{cases}
\end{equation}
This identity follows directly from Lemma~\ref{lem:series_alt}.
Evaluating $S^{(\pm)}_{N,p}(-x)$ using equation~(\ref{eq:series_alt}) and applying relations~(\ref{eq:jc0_mx}) and~(\ref{eq:js0_mx}), one verifies that the result coincides with $(-1)^p S^{(\pm)}_{N,p}(x)$ as given by equation~(\ref{eq:series}).

\subsection{Bilateral series}
We also use Lemma~\ref{lem:series_alt} to compare the one-sided series with bilateral series.
A generic expression for bilateral Bessel $J_\nu$ series was derived in~\cite{Sung2022}:
\begin{align}
\label{eq:sung_1}
\sum_{m=-\infty}^{\infty} J_{N m+p}(x) &= \frac{1}{N}\sum_{q=0}^{N-1} e^{ix\sin(2\pi q/N)}e^{-i2\pi pq/N }\\
\label{eq:sung_2}
\sum_{m=-\infty}^{\infty} (-1)^m J_{N m+p}(x) &= \frac{1}{N} \sum_{q=0}^{N-1}e^{ix\sin((2q+1)\pi/N)}e^{-i(2q+1)\pi p/N}.
\end{align}
Note that equations~(\ref{eq:sung_1})--(\ref{eq:sung_2}) can be further simplified by observing that only the real part of the complex exponential sum survives:
\begin{equation}
\label{eq:series_bi}
\sum_{m=-\infty}^{\infty} \qty(\pm 1)^m J_{N m+p}(x) = \frac{1}{N}\sum_{q=0}^{N-1} \cos(x a^{(\pm)}_q - \frac{2\pi}{N} p c^{(\pm)}_q).
\end{equation}

Notably, the bilateral series result is free of special functions and involves only trigonometric functions.
Equation~(\ref{eq:series_bi}) can also be derived independently by means of the following lemma.
\begin{lemma}[Bilateral series]
\label{lem:series_bi}
For any $N \in \mathbb{N}$, $1 \le p < N$ equation~(\ref{eq:series_bi}) holds.
\end{lemma}

\begin{proof}
Since the bilateral series is known to converge, we decompose it as
\begin{multline}
\sum_{m=-\infty}^{\infty} \qty(\pm 1)^m J_{N m+p}(x) = \sum_{m=0}^{\infty} \qty(\pm 1)^m J_{N m+p}(x) \\
\pm \qty(-1)^{N-p} \sum_{m=0}^{\infty} \qty(-1)^{Nm} \qty(\pm 1)^m J_{N m+\qty(N-p)}(x),
\end{multline}
which follows from the general Bessel function property $J_{-\nu}(x) = \qty(-1)^\nu J_{\nu}(x).$
The first term is evaluated using equation~(\ref{eq:series}), and the second term using equation~(\ref{eq:series_alt}).
Note that the second term has $N-p$ in place of $p$, and its contribution can be further simplified using the periodicity of the cosine and sine functions:
\begin{equation}
\cos(\frac{2\pi}{N}\qty(N-p)c^{(\pm)}_k + \phi) = \pm \cos(\frac{2\pi}{N}pc^{(\pm)}_k - \phi),
\end{equation}
\begin{equation}
\sin(\frac{2\pi}{N}\qty(N-p)c^{(\pm)}_k + \phi) = \pm \sin(\phi - \frac{2\pi}{N}pc^{(\pm)}_k).
\end{equation}
Upon adding the two expressions, the incomplete Lipschitz-Hankel integral terms cancel, while the cosine terms combine to yield equation~(\ref{eq:series_bi}).
\end{proof}

\subsection{Comparison with existing literature}
Tables~\ref{tab:neumann_1} and~\ref{tab:neumann_2} include references to known sources for the special cases listed.
Note that the results are consistent with the different notation used for $\JC0(0, x)$ across various authors.
The more specialized tables~\cite{Prudnikov1986} contain two additional results that fall outside the scope of our tables: (5.7.1)~.21 and (5.7.1)~.22.
These can also be obtained from equation~(\ref{eq:series}) up to differences due to trigonometric transformations:
\begin{equation}
\sum_{m=0}^{\infty} J_{8m}(x) = \frac{1}{8} \qty(\cos(x)+2 \cos(\frac{x}{\sqrt{2}})+4 J_0(x)+1),
\end{equation}
\begin{equation}
\sum_{m=0}^{\infty} J_{8m+4}(x) = \frac{1}{8} \qty(\cos(x)-2 \cos(\frac{x}{\sqrt{2}})+1),
\end{equation}
where~\cite{Prudnikov1986} uses the substitution $\cos(x)+1 = 2\cos^2(\frac{x}{2})$.

\subsection{Numerical evaluation}
We also computed $\sum_{m=0}^{\infty} \qty(\pm 1)^m J_{N m+p}(x)$ numerically for fixed values of $N$ and $p$ using both direct summation and equation~(\ref{eq:series}).\footnote{\url{https://github.com/matwey/neumann_series}}
The multiprecision library {\tt mpmath}~\cite{mpmath} for Python was used for this purpose.
The numerical results were found to agree within six decimal digits.
Since incomplete Lipschitz-Hankel integrals of Bessel functions $\JC0(a,x)$ and $\JS0(a,x)$ are not readily available in {\tt mpmath},
the Kampé~de~Fériet representation was employed~\cite{Miller1989}:
\begin{equation}
J_{c_{\mu,\nu}}(a, z) = \frac{z^{1+\mu+\nu}}{2^\nu (1+\mu+\nu) \Gamma(\nu+1)} 
F{}^{1:0;0}_{1:1;1}\left(
\setlength{\arraycolsep}{0pt}
\begin{array}{c@{{}:{}}c@{{};{}}c@{{};}}
\frac{1+\mu+\nu}{2} & - & - \\
\frac{3+\mu+\nu}{2} & 1+\nu & \frac{1}{2}
\end{array}
-\dfrac{z^2}{4}, -\dfrac{(a z)^2}{4}
\right),
\end{equation}
\begin{equation}
J_{s_{\mu,\nu}}(a, z) = \frac{a z^{2+\mu+\nu}}{2^\nu (2+\mu+\nu) \Gamma(\nu+1)}
F{}^{1:0;0}_{1:1;1}\left(
\setlength{\arraycolsep}{0pt}
\begin{array}{c@{{}:{}}c@{{};{}}c@{{};}}
\frac{2+\mu+\nu}{2} & - & - \\
\frac{4+\mu+\nu}{2} & 1+\nu & \frac{3}{2}
\end{array}
-\dfrac{z^2}{4}, -\dfrac{(a z)^2}{4}
\right),
\end{equation}
where
\begin{equation}
J_{c_{\mu,\nu}}(a, z) = \int_0^z \cos(a t) t^\mu J_\nu(t) dt,
\end{equation}
\begin{equation}
J_{s_{\mu,\nu}}(a, z) = \int_0^z \sin(a t) t^\mu J_\nu(t) dt.
\end{equation}
This is probably not the most efficient method for computing incomplete Lipschitz-Hankel integrals of Bessel functions, but it is adequate for the present purposes when used with a multiprecision library.

\section{Conclusion}
\label{sec:conclusion}

We have derived generic expressions for the one-sided series
$\sum_{m=0}^{\infty} (\pm 1)^m J_{Nm+p}(x)$
valid for arbitrary positive integer $N$ and integer $0 \le p < N$.
The derivation proceeds by reformulating the partial sums as the solution of a nonhomogeneous linear ODE system with constant coefficients, whose matrix exponential is evaluated via eigendecomposition.
The resulting expressions are given in terms of incomplete Lipschitz-Hankel integrals of Bessel functions of the first kind.

Special cases of these series have appeared previously in reference tables~\cite{Gradshteyn2014,Abramowitz1965,Prudnikov1986}, and we have verified the consistency of our general formula with those entries.
The relation between one-sided and bilateral series was also established, recovering the result of~\cite{Sung2022} as a special case.

The Kamp\'e de F\'eriet representation used in this work, while suitable for multiprecision computation, may not be optimal for large-scale applications; fortunately, specific algorithms for numerical evaluation of the incomplete Lipschitz-Hankel integrals are readily available in the literature.

The other direction for future investigation is to check reducibility of the incomplete Lipschitz-Hankel integrals for $a_k=\sin(\frac{2\pi}{N}k)$.
Here we consider only a subset of all possible $a$ values, so additional reducibility may potentially be available.

\printbibliography

\end{document}